\documentclass[12pt]{amsart}

\usepackage[all]{xy}
\usepackage{graphics,bm}
\usepackage{color}
\usepackage[T1]{fontenc}
\usepackage{parskip}
\usepackage[colorlinks]{hyperref}
\usepackage{enumitem}
\usepackage{comment}
\usepackage{amsfonts}
\usepackage{mathrsfs}
\usepackage{amsmath,amssymb, mathtools}
\usepackage[shortcuts]{extdash}

\usepackage{tikz,tikz-cd}
\usetikzlibrary{patterns}
\usetikzlibrary{decorations.markings}

\usepackage[colorinlistoftodos,prependcaption]{todonotes}

\newcommand{\B}[1]{{\mathbf #1}}
\newcommand{\C}[1]{{\mathcal #1}}

\newcommand\mydots{\hbox to 0.8em{.\hss.\hss.}\thinspace}
\newcommand{\overbar}[1]{\mkern 1.5mu\overline{\mkern-1.5mu#1\mkern-1.5mu}\mkern 1.5mu}

\newcommand*{\defeq}{\mathrel{\vcenter{\baselineskip0.5ex \lineskiplimit0pt
	                     \hbox{\scriptsize.}\hbox{\scriptsize.}}}%
			                          =}
\newcommand{\n}{|\hspace{-1px}|}

\newcommand{\rH}{\overbar{\OP{H}}}
\newcommand{\rEH}{\overbar{\OP{EH}}}
\newcommand{\Gb}{\Gamma_{\hspace{-2px}b}}
\newcommand{\EGb}{\OP{E}\hspace{-2px}\Gamma_{\hspace{-2px}b}}
\newcommand{\rGb}{\overbar{\Gamma}_{\hspace{-2px}b}}
\newcommand{\rEGb}{\overbar{\OP{E}\hspace{-2px}\Gamma}_{\hspace{-2px}b}}

\newtheorem{theorem}{Theorem}[section]
\newtheorem{theorem*}{Theorem}

\newtheorem{lemma}{Lemma}[section]
\newtheorem{proposition}[lemma]{Proposition}

\newtheorem*{question*}{Question}

\theoremstyle{definition}

\newtheorem{remark}[lemma]{Remark}
\newtheorem{question}[lemma]{Question}
\newtheorem*{remark*}{Remark}
\newtheorem*{remarks*}{Remarks}
\newtheorem*{corollary*}{Corollary}

\numberwithin{figure}{section}
\numberwithin{table}{section}
\numberwithin{equation}{section}

\def\B{\mathbf}
\newcommand{\OP}{\operatorname}

\newcommand{\wt}{\widetilde}

\begin{document}

\title[Bounded cohomology of transformation groups]
{Bounded cohomology of transformation groups}
\author{Michael Brandenbursky and Micha\l\ Marcinkowski}
\address{Ben Gurion University of the Negev, Israel}
\email{brandens@bgu.ac.il}
\address{Institute of Mathematics, Polish Academy of Sciences, Wroc\l aw, Poland}
\email{marcinkow@math.uni.wroc.pl} 
\subjclass[2010]{20,51}

\begin{abstract}  
Let $M$ be a~complete connected Riemannian manifold of finite volume. 
We present a~new method of constructing classes 
in bounded cohomology of transformation groups such as
$\OP{Homeo}_0(M,\mu)$, $\OP{Diff}_0(M,\OP{vol})$ and 
$\OP{Symp}_0(M,\omega)$.
As an application we show that for many manifolds (in particular for hyperbolic surfaces)
the $3^{rd}$ bounded cohomology of these groups is infinite dimensional. 
\end{abstract} 

\maketitle

Let $M$ be a~complete connected Riemannian manifold with empty boundary and of finite volume, and
let $\mu$ be the induced (by the Riemannian metric) measure on $M$.
Denote by $\OP{Homeo}(M,\mu)$ the group of all $\mu$-preserving compactly supported homeomorphisms of $M$, and
by $\OP{Homeo}_0(M,\mu)$ the connected component of the identity of $\OP{Homeo}(M,\mu)$. 
For a group $G$ we denote by $\OP{H}_b^\bullet(G)$ the bounded cohomology of $G$.
In this paper we define and study the homomorphism:
$$
\Gb \colon \OP{H}_b^\bullet(\pi_1(M)) \to \OP{H}_b^\bullet(\OP{Homeo}_0(M,\mu)).
$$

Note that $\OP{H}_b^\bullet(\OP{Homeo}_0(M,\mu))$ is the bounded cohomology of the \textit{discrete} group $\OP{Homeo}_0(M,\mu)$,
see Section~\ref{s:prelimin} for definitions. 
%As a topological group $\OP{Homeo}_0(M,\mu)$ might be contractible, thus its cohomology is trivial 
%(e.g., if $M$ is a hyperbolic surface, then $\OP{Homeo}_0(M,\mu)$ with $C_0$-topology is contractible). 

%Let us note, that $\OP{Homeo}_0(M,\mu)$ has no non-abelian quotients, thus one can-not pull-back cohomology classes from the quotients of $\OP{Homeo}_0(M,\mu)$.

The map $\Gb$ can be seen as a~generalization of the construction 
given by Gambaudo-Ghys \cite{MR2104597} and Polterovich \cite{MR2276956}, which in particular
was extensively used in the study of several conjugacy invariant norms on transformation groups 
\cite{MR3391653, MR3044593, BKS, eq, 1405.7931, MR2104597}.
This construction was restricted only to non-trivial homogeneous quasimorphisms, 
that is, to a~certain subspace of the second
bounded cohomology. In this paper we deal with bounded cohomology in all dimensions. 

The map $\Gb$ can be also defined for the exact and the reduced bounded cohomology. 
Moreover, another advantage of $\Gamma_b$ is that it has a~counterpart that works for the \textit{ordinary cohomology} in all dimensions.
However, the case of the ordinary cohomology seems to be less natural for our construction and harder to work with.

Let us note, that every homomorphism $h \colon \OP{Homeo}_0(M,\mu) \to G$ has abelian image. 
Since bounded cohomology of abelian groups is trivial, the induced map $h^* \colon \OP{H}^\bullet_b(G) \to \OP{H}^\bullet_b(\OP{Homeo}_0(M,\mu))$ is trivial.  
However, $\OP{Homeo}_0(M,\mu)$ admits an interesting measurable cocycle (homomorphisms are particular examples of measurable cocycles), 
which we construct in Section \ref{s:def} and use to define~$\Gb$. 

By definition $\OP{Homeo}_0(M,\mu)$ preserves the measure~$\mu$. 
This property is essential, since it is known that $\Gb$ can not factor through 
larger groups such as $\OP{Homeo}_0(M)$, see Remark \ref{r:extend}.

We use the map $\Gb$ in order to show non-triviality of the third bounded cohomology 
of many transformation groups $\C T_M<\OP{Homeo}_0(M,\mu)$. 
Below we describe three main families of examples we focus on:

\begin{itemize}
	\item $M$ is a complete Riemannian manifold of finite volume, $\mu$ is the measure 
	induced by the Riemannian metric and $\C T_M=\OP{Homeo}_0(M,\mu)$. 
	
	\item $M$ is a complete Riemannian manifold of finite volume, $\OP{vol}$ 
	is the volume form induced by the Riemannian metric 
	and $\C T_M = \OP{Diff}_0(M,\OP{vol})$, i.e., it is 
	the identity component of the group of compactly supported volume-preserving diffeomorphisms of $M$. 
	
	\item $(M,\omega)$ is a~symplectic manifold admitting  
	a~complete Riemannian metric, and $\C T_M=\OP{Symp}_0(M,\omega)$, i.e., it is
	the identity component of the group of 
	compactly supported symplectomorphisms of $M$.
	We think of $\C T_M$ as a~subgroup of $\OP{Homeo}_0(M,\mu)$, 
	where $\mu$ is induced by the volume form $\omega^{\frac{1}{2}\OP{dim}M}$ and it is finite.  
\end{itemize}

Let $G$ be a~group and let $\rEH_b^\bullet(G)$ denote the reduced 
exact bounded cohomology of $G$ (see Section \ref{s:prelimin}).
Denote by $\OP{dim}\rEH_b^\bullet(G)$ the real linear dimension of $\rEH_b^\bullet(G)$ 
and by $\OP{F}_2$ the free group of rank $2$.

\begin{theorem}\label{t:surjects.on.F} Let $M$ be as above and
suppose that $\pi_1(M)$ surjects onto $\OP{F}_2$. Then for every $n$ there exists a monomorphism 
$$
\overbar{\OP{EH}}_{b}^n(\OP{F}_2) \hookrightarrow \overbar{\OP{EH}}_{b}^n(\C T_M),
$$
where $\C T_M $ is either $\OP{Homeo}_0(M,\mu)$, or $\OP{Diff}_0(M,\OP{vol})$, or $\OP{Symp}_0(M,\omega)$.
\end{theorem}

Let $\pi_M = \pi_1(M)/Z(\pi_1(M))$, 
where $Z(\pi_1(M))$~is the center of $\pi_1(M)$.

\begin{theorem}\label{t:hyp.embed} Let $M$ be as above and
suppose that $\OP{F}_2 \times K$ embeds hyperbolically into $\pi_M$, where $K$ is a~finite group. 
Then for every $n$
$$
\OP{dim} \overbar{\OP{EH}}_{b}^n(\C T_M) \geq \OP{dim} \overbar{\OP{EH}}_{b}^n(\OP{F}_2),
$$
where $\C T_M $ is either $\OP{Homeo}_0(M,\mu)$, or $\OP{Diff}_0(M,\OP{vol})$, or $\OP{Symp}_0(M,\omega)$.
\end{theorem}	

The notion of hyperbolic embedding was defined in \cite{MR3589159}. 
Note that the statement of Theorem \ref{t:hyp.embed} holds if $\pi_1(M)$ is a~center-free acylindrically hyperbolic group, 
since such a~group always admits a~hyperbolically embedded $\OP{F}_2 \times K$, 
see \cite[Theorem 1.2]{MR3430352} and \cite[Theorem 2.24]{MR3589159}. 
Examples of acylindrically hyperbolic groups include:
\begin{enumerate}
\item non-elementary hyperbolic groups and relatively hyperbolic groups,
\item mapping class groups of hyperbolic surfaces and outer automorphism groups of non-abelian free groups, 
\item most $3$-manifolds groups,
\item right angled Artin groups that are not direct products.
\end{enumerate}

Theorems \ref{t:surjects.on.F} and \ref{t:hyp.embed} imply that 
bounded cohomology of $\C T_M$ is non trivial in dimension $3$,
which is a new result. Indeed, due to a theorem of Soma \cite{soma}, 
the dimension of the third reduced exact bounded cohomology of $\OP{F}_2$ is continuum, and
hence we obtain the following

\begin{corollary*}
If the conditions of Theorem \ref{t:surjects.on.F} or Theorem \ref{t:hyp.embed} hold, then 
$$
\OP{dim} \rEH^3_{b}(\C T_M) \geq 2^{\aleph_0}.
$$
\end{corollary*}

Unfortunately we are not able to obtain a similar result for $n>3$, 
since nothing is known about the dimension of $\rEH^n_{b}(\OP{F}_2)$.
We also would like to mention that versions of Theorem \ref{t:surjects.on.F} 
and Theorem~\ref{t:hyp.embed} hold in a~more general setting, 
see Remark \ref{r:general}.

We construct non-trivial classes in $\rEH_b^3(\C T_M)$ 
out from non-trivial classes in $\rEH_b^3(\OP{F}_2)$.
There is a nice family of elements in $\rEH_b^3(\OP{F}_2)$, that are represented by cocycles 
which are defined in terms of volumes of geodesic simplices in the hyperbolic $3$-space $\B H^3$. 
The elements of $\rEH_b^3(\C T_M)$, which are constructed from such classes in Theorem~\ref{t:surjects.on.F}, 
have similar geometrical description, see Remark \ref{r:int.volume}.

We would like to mention that elements of $\OP{H}_b^n(\C T_M)$ can be interpreted as bounded characteristic 
classes of foliated $M$-bundles with holonomy in $\C T_M$. 
Such classes were studied in the case of ordinary cohomology 
for groups $\OP{Diff}(M,\OP{vol})$ and $\OP{Symp}(M,\omega)$, see e.g., 
\cite{kontsevich, kot.mor.sympl, kot.mor.vol, sam.stability, rez99, rez97}. 
The classes constructed in this paper are of entirely different nature.  

\textbf{Acknowledgments}.
Both authors were supported by SFB 1085 ``Higher Invariants'' funded by Deutsche Forschungsgemeinschaft. 
The second author was supported by grant Sonatina 2018/28/C/ST1/00542 funded by Narodowe Centrum Nauki.

\section{Preliminaries}\label{s:prelimin}
\subsection{Bounded cohomology}
Bounded cohomology was defined in a seminal paper of Gromov \cite{gromov82}. 
Let us recall basic definitions. 

Let $G$ be a~group. 
A function $c \colon G^{n+1} \to \mathbb{R}$ is called homogeneous, 
if for every $h \in G$ and every $g_0,\ldots,g_n\in G$ we have 
$$c(g_0h,\ldots,g_nh) = c(g_0,\ldots,g_n).$$ 
The space of bounded $n$-cochains is defined by
$$\OP{C}^n_b(G) = \{c \colon G^{n+1} \to \mathbb{R}~|~c~\text{is homogeneous and bounded}\}.$$
Let $d$ be the ordinary coboundary operator $d_n \colon  \OP{C}^n_b(G) \to \OP{C}^{n+1}_b(G)$. 
The \textbf{bounded cohomology} of $G$, denoted $\OP{H}^\bullet_b(G)$, is the homology of the chain complex $\{\OP{C}^n_b(G),d_n\}$.
Note that $\OP{C}_b^n(G)$ is a~subcomplex of the space of all homogeneous cochains, hence we have a 
homomorphism $\OP{H}_b^n(G) \to \OP{H}^n(G,\mathbb{R})$ called the \textbf{comparison map}. 
The \textbf{exact bounded cohomology}, denoted $\OP{EH}_b^n(G)$, 
is defined to be the kernel of the comparison map. 
A bounded class belongs to $\OP{EH}_b^n(G)$ if it is the coboundary of a cochain. 
It is non-trivial, if it is not the coboundary of a~\textit{bounded} cochain.

On $\OP{C}^n_b(G)$ we have the supremum norm denoted by $\n\hspace{-3px}\cdot\hspace{-3px}\n_{sup}$.
This norm induces a~semi-norm on $\OP{H}^n_b(G)$, i.e.,
if $c \in \OP{H}_b^n(G)$, then 
$$\n c \n = \OP{min}\{\n a \n_{sup}~|~[a] = c\}.$$
Let 
$$\OP{N}^n(G) = \{c \in \OP{H}_b^n(G)~|~\n c \n = 0\}.$$
Since $\n\hspace{-3px}\cdot\hspace{-3px}\n$ is a seminorm, $\OP{N}^n(G)$ is a linear subspace of $\OP{H}_b^n(G)$.
The \textbf{reduced bounded cohomology} is defined by 
$$\rH^n_b(G) = \OP{H}_b^n(G)/\OP{N}^n(G).$$
Note that $\rH^n_b(G)$ equipped with the induced norm, again denoted by $\n\hspace{-3px}\cdot\hspace{-3px}\n$,
is a Banach space. 
%Let $q \colon \OP{H}_b^n(G) \to \rH_b^n(G)$ be the quotient map. 
The \textbf{exact reduced bounded cohomology} is a Banach subspace of $\rH^n_b(G)$ defined by 
$$\rEH_b^n(G)~=~\OP{EH}^n_b(G)/(\OP{EH}^n_b(G) \cap \OP{N}^n(G)).$$

\subsection{Measurable cocycles and bounded cohomology}\label{s:induction}

Let $G$ be a~topological group, $H$ a~discrete group and $(X,\mu)$ a~measurable space. 
Suppose that $G$ acts on $X$ by measure preserving homeomorphisms. 
A map $\gamma \colon G \times X \to H$ is called a~\textbf{measurable cocycle}, if
for every $g \in G$ the map $\gamma(g,\cdot)$ is measurable and for all $g_1,g_2 \in G$ and for  
almost all $x \in X$ we have
$$\gamma(g_1g_2,x) = \gamma(g_1,g_2(x))\gamma(g_2,x).$$ 

Note that if $X$ is a point, then $\gamma$ is a homomorphism from $G$ to $H$. 

We show that a measurable cocycle $\gamma$ induces a homomorphism on bounded cohomology
$$\OP{Ind}_b(\gamma) \colon \OP{H}_b^n(H) \to \OP{H}_b^n(G).$$
To define $\OP{Ind}_b(\gamma)$, we first define a map 
$\OP{Ind}'_b(\gamma) \colon \OP{C}^n_b(H) \to \OP{C}^n_b(G)$ 
by the formula:
$$
\OP{Ind}'_b(\gamma)(c)(g_0,g_1,\ldots,g_n) = 
\int_{M} c(\gamma(g_0,x),\gamma(g_1,x),\ldots,\gamma(g_n,x)) d\mu(x),
$$
where $c \in \OP{C}^n_b(H)$. 

The next proposition shows, that $\OP{Ind}'_b(\gamma)$ is well-defined and commutes with $d$. 
Hence we can define 
$$\OP{Ind}_b(\gamma)([c]) = [\OP{Ind}'_b(\gamma)(c)],$$ 
where $c \in  \OP{C}_b^n(H)$ and $[c] \in \OP{H}_b^n(H)$. 

\begin{proposition}\label{p:well.def}
Let $G$ be a group, $c \in \OP{C}_b^n(H)$ and let $g_0,g_1,\ldots,g_n~\in~G$. Then the following holds: 
$$x \to c(\gamma(g_0,x),\gamma(g_1,x),\ldots,\gamma(g_n,x))$$
is a $\mu$-measurable function on $X$, 
$\OP{Ind}'_b(\gamma)$ commutes with the coboundary $d$, 
and $\OP{Ind}'_b(\gamma)(c)$ is a homogeneous cocycle.
\end{proposition}

\begin{proof}
The function $X \to H^n$ defined by 
$$x \to (\gamma(g_0,x),\gamma(g_1,x),\ldots,\gamma(g_n,x))$$
is $\mu$-measurable and $c$ is continuous. 
Thus their composition is $\mu$-measurable. 

Commutativity of $\OP{Ind}'_b(\gamma)$ and $d$ 
follows directly from the definition of $\OP{Ind}'_b(\gamma)$. 

Let $h ,g_0,\ldots \in G$, we have:
		\begin{align*}
			\OP{Ind}'_b(\gamma)(c)(g_0h,\ldots)& = \int_{X} c(\gamma(g_0h,x),\ldots) d\mu(x)\\
			&=  \int_{X} c(\gamma(g_0,h(x))\gamma(h,x),\ldots) d\mu(x)\\
			&=  \int_{X} c(\gamma(g_0,h(x)),\ldots) d\mu(x)\\
			&=  \int_{X} c(\gamma(g_0,x),\ldots) dh^*\mu(x)  = \OP{Ind}'_b(\gamma)(c)(g_0,\ldots).
		\end{align*}
The above equalities follow from the cocycle condition, the homogeneity of $c$ 
and the fact that $\mu$ is $h$-invariant. 
\end{proof}

We conclude this section by noting 
that maps similar to $\OP{Ind}_b(\gamma)$ were used in the study of ordinary 
continuous cohomology of Lie groups \cite{guichardet} and 
geometry of solvable and amenable groups \cite{sauer,shalom}. 

\section{Definition of $\Gb$}\label{s:def}
In this section we construct the homomorphism
$$\Gb\colon \OP{H}_b^\bullet(\pi_1(M)) \to \OP{H}_b^\bullet(\OP{Homeo}_0(M,\mu)).$$

\subsection{The cocycle}\label{s:cocycle}
Denote by $\C H_M=\OP{Homeo}_0(M)$ the identity component of the group of compactly supported 
homeomorphisms of $M$. 
Let $z \in M$ be a basepoint and 
let $\C H_{M,z}<\C H_M$ be the subgroup of all homeomorphisms in $\C H_M$ that fix $z$. 
Consider the following fiber bundle 
$$\C H_{M,z} \to \C H_M \xrightarrow{ev_z} M,$$
where $ev_z$ is the evaluation map at the basepoint $z$, i.e., $ev_z(g) = g(z)$ for $g \in \C H_M$,
and $\C H_{M,z}$ is the fiber of $ev_z$. 
%The advantage of $\C H_{M,z}$ over $\C H_M$ is that the latter has no non-abelian quotients, 
%whereas $\C H_{M,z}$ maps onto $\pi_{M} = \pi_1(M)/Z(\pi_1(M))$.

%has an interesting quotient: 
%it maps onto the group of connected components $\pi_0(\C H_{M,z})$ which
%is closely related to $\pi_1(M)$. Recall that we denoted $\pi_{M} = \pi_1(M)/Z(\pi_1(M))$, 
%where $Z(\pi_1(M))$ is the center of $\pi_1(M)$.

Let $ev_z^* \colon \pi_1(\C H_M) \to \pi_1(M,z)$ be induced by $ev_z$. 

\begin{proposition}\label{p:delta}
The image of $ev_z^*$ is contained in the center of $\pi_1(M,z)$.
\end{proposition}

\begin{proof}
%We claim that $\OP{im}(ev_z^*) \subset Z(\pi_1(M,z))$. 
Let $g_t$, $t\in S^1$, be a~loop in $\C H_M$ based at the identity and $[g_t] \in \pi_1(\C H_M)$. 
Then $ev_z^*([g_t])$~is a~loop based at $z$ represented by $g_t(z)$.
Let $l_s$, $s\in S^1$, be an arbitrary loop in $M$ based at $z$.
The image of the map $S^1 \times S^1 \to M$ given by $(t,s) \to g_t(l_s)$
contains $g_t(z)$ and $l_s$, thus these loops commute in~$\pi_1(M,z)$.
\end{proof}

Let us consider the long exact sequence of homotopy groups of the fibration
$\C H_{M,z} \to \C H_M \xrightarrow{ev_z} M$:
$$\pi_1(\C H_M) \xrightarrow{ev_z^*} \pi_1(M,z) \to \pi_0(\C H_{M,z}) \to \pi_0(\C H_M) = 1.$$

It follows from this exact sequence that $\pi_0(\C H_{M,z}) \cong \pi_1(M,z)/\OP{im}(ev_z^*)$. 
We define 
$$\delta \colon \C H_{M,z} \to \pi_M$$ 
to be the composition of the map
$$\C H_{M,z} \to \pi_0(\C H_{M,z}) \cong \pi_1(M,z)/\OP{im}(ev_z^*)$$ and
the quotient map $\pi_1(M,z)/\OP{im}(ev_z) \to \pi_{M}$.

Let $\C H_{M,\mu} = \OP{Homeo}_0(M,\mu)$ and let $s \colon M \to \C H_M$ be a~measurable section of $ev_z$, 
i.e., $ev_z \circ s_x = x$ for almost all $x \in M$.
We define a~measurable cocycle
$$\gamma_s \colon \C H_{M,\mu} \times M \to \pi_M$$
by the following formula
$$\gamma_s(g,x) = \delta(s_{g(x)}^{-1} \circ g \circ s_x).$$

It follows immediately from the definition that $\gamma_s$ satisfies the cocycle condition.

\subsection{Example of a~section \textit{s}}\label{s:example}
Let us consider the following set:
$$ D = \OP{int}(\{~x \in M~|~\text{there exists a~unique geodesic between}~z~\text{and}~x~\}).$$
The set $M\setminus D$ is called the cut locus of $M$. The Hausdorff dimension of $M\setminus D$ is at most
$\OP{dim}(M)-1$, see \cite{cut.locus}. Thus $\mu(M\setminus D) = 0$, and $\mu(D) = \mu(M)$. 
Let $x \in D$ and let $s_x$ be a~point-pushing map that transports $z$ to $x$ along the geodesic.
We choose these point-pushing maps such that they define a continuous section $s \colon D \to \C H_M$.
Since we regard $s$ as a~measurable map it is enough to define $s$ on a full measure subset of $M$. 

Let us now take a~closer look at the cocycle $\gamma_s$ defined by such $s$. 
Let $g \in \C H_M$ and $x \in D \cap g^{-1}(D) \subset M$.
The element $\gamma_s(g,x) \in \pi_{M}$ has a~simple geometrical interpretation. 
It can be constructed as follows: let $g_t$, $t \in [0,1]$, be any isotopy in $\C H_M$~connecting the identity to $g$. 
Let $\alpha$ be the concatenation of the geodesic from $z$ to $x$, the path $g_t(x)$, $t \in [0,1]$, 
and the geodesic from $g(x)$ to $z$. 
It is clear that $\alpha$ is a~loop based at $z$. Denote its homotopy class by $[\alpha]\in\pi_1(M)$. 
The element $\gamma_s(g,x)\in \pi_{M}$ is a coset represented by $[\alpha]$.   
 
\subsection{The definition of $\Gb$}
Let $\gamma_s \colon \C H_{M,\mu} \times M \to \pi_M$ be a~measurable cocycle
given by a~measurable section $s \colon M \to \C H_M$. 
We define
$$\Gb = \OP{Ind}_b(\gamma_s) \colon \OP{H}_b^\bullet(\pi_M) \to \OP{H}_b^\bullet(\C H_{M,\mu}).$$
Note that the quotient map $\pi_1(M) \to \pi_{M}$ has an abelian kernel. 
It follows from \cite[Section 3.1]{gromov82} that the induced map
$$\OP{H}_b^\bullet(\pi_{M}) \to \OP{H}_b^\bullet(\pi_1(M))$$
in an isometric isomorphism, hence $\Gb$ takes the following form:
$$\Gb \colon \OP{H}_b^\bullet(\pi_1(M)) \to \OP{H}_{b}^\bullet(\C H_{M,\mu}).$$ 

Let $\C T_{M}$ be a subgroup of $\C H_{M,\mu}$. The composition of $\Gb$ with the restriction map
$\OP{H}_b^\bullet(\C H_{M,\mu}) \to \OP{H}_b^\bullet(\C T_M)$ gives
$$\Gb(\C T_{M}) \colon \OP{H}_b^\bullet(\pi_1(M)) \to \OP{H}_{b}^\bullet(\C T_{M}).$$
Usually we abuse the notation and write $\Gb$ instead of $\Gb(\C T_{M})$.  

\subsection{Standard cohomology and exact bounded cohomology}\label{s:standard}
Let $s$ be the section defined in Subsection \ref{s:example} and $\gamma = \gamma_s$.
We show that for such $\gamma$, a similar map can be defined for the ordinary cohomology. 
Note that in the case of the ordinary cohomology the definition 
requires more effort then in the case of the bounded cohomology.
The reason is that now cocycles are not bounded and we need to show that the integral exists. 
%In order to show integrability we need to choose carefully the section $s$.  

Let $x \in M$, $c \in \OP{C}^n(\pi_M)$ and $g_0,\ldots,g_n \in \C H_{M,\mu}$.
Consider the function $x \to c(\gamma(g_0,x),\ldots,\gamma(g_n,x))$.
It follows from Proposition \ref{p:well.def} that it is measurable. 
Integrability follows from Lemma \ref{l:fin.many.loops} below, since
every measurable function with essentially finite image is integrable. 
Note that Lemma \ref{l:fin.many.loops} holds only for sections described in Subsection \ref{s:example},
and in general not every section induces an integrable function. 
We define
$$\OP{Ind}'(\gamma)(c)(g_0,g_1,\ldots,g_n) = 
\int_{M} c(\gamma(g_0,x),\gamma(g_1,x),\ldots,\gamma(g_n,x)) d\mu(x).$$
It follows immediately from Proposition \ref{p:well.def} that $d$ 
and $\OP{Ind}'(\gamma)$ commute and that $\OP{Ind}'(\gamma)(c)$ is homogeneous.
Hence $\OP{Ind}'(\gamma)$ induces
$$
\Gamma \colon \OP{H}^\bullet(\pi_M) \to \OP{H}^\bullet(\C H_{M,\mu}).
$$

Let $\C T_M \leq \C H_{M,\mu}$ and $i^* \colon \OP{H}^\bullet(\C H_{M,\mu}) \to \OP{H}^\bullet(\C T_M)$ is induced by the inclusion. Define
$$
\Gamma(\C T_M) \defeq i^* \circ \OP{Ind}'(\gamma) \colon \OP{H}^\bullet(\pi_M) \to \OP{H}^\bullet(\C T_M).
$$

Let us discuss Lemma \ref{l:fin.many.loops}. 
Since in this lemma measure preservation does not play any role, 
it is natural to extend the cocycle $\gamma$ to $\OP{Homeo}_0(M)$.
Namely, let
$$
\gamma' \colon \OP{Homeo}_0(M) \times M \to \pi_M,
$$
be defined by the same formula as $\gamma$, i.e, 
$$\gamma'(g,x) = \delta(s_{g(x)}^{-1} \circ g \circ s_x),$$
where $g \in \OP{Homeo}_0(M)$. Let $X$ be a~measurable space and $f \colon X \to Y$ be a~measurable function. 
We say that $f$ has \textbf{essentially finite image}, 
if there exists a~full measure subset $Z \subset X$, such that $f$ has a finite image in $Z$. 

\begin{lemma}\label{l:fin.many.loops}
For every $f\in\OP{Homeo}_0(M)$ the map $\gamma'(f,\cdot\hspace{1px})\colon M\to\pi_M$ has essentially finite image. 
\end{lemma}
\begin{proof}
Let $f\in\OP{Homeo}_0(M)$ and $\{f_t\}$ be an isotopy between the identity and $f$.
The union of the supports $\bigcup_{t\in[0,1]}\OP{supp}(f_t)$ is a~
compact subset of $M$. Recall that $M$ admits a~complete Riemannian metric.
Hence there exists $r>0$ such that the geodesic ball $B_r(z)$ of 
radius $r$ centered at $z$ contains $\bigcup_{t\in[0,1]}\OP{supp}(f_t)$.
Note that for each $x$ lying in the full measure subset of  $M\setminus B_r(z)$ the element $\gamma'(f,x)$ 
is trivial in $\pi_M$. Hence it is enough to show that the
set $\{\gamma'(f,x)\}$, where $x$ belongs to the full measure
subset of $B_r(z)$, is finite in $\pi_1(B_r(z), z)$. 
We consider $\gamma'(f,x)$ as an element of $\pi_1(M,z)$. 

The group $\OP{Homeo}_0(M)$ admits a~fragmentation property with 
respect to any open cover of $M$, see \cite[Corollary 1.3]{MR0283802}.
Hence the ball $B_r(z)$ can be covered by finite number of balls $B_i$ with the following property: 
$f$ can be written as a product of homeomorphisms $h_i$
such that the support of $h_i$ lies in $B_i$. Since $M$ is a~smooth manifold, for each $i$ there 
exits a~smooth ball $B'_i$, such that it is $\epsilon$-close to $B_i$ and such that 
it is $\epsilon$-homotopic to $B_i$, see smooth approximation theorem \cite[Theorem 2.11.8]{MR1224675}.
Note that $\gamma'(f,x)$ satisfies the cocycle condition. It means that
$$\gamma'(f,x)=\gamma'(h_1\circ\ldots\circ h_n, x)=\gamma'(h_1,h_2\circ\ldots\circ h_n(x))\ldots\gamma'(h_n,x).$$
Hence it is enough to prove that the set $\{\gamma'(h_i,x)\}$ where $x$ belongs to the full
measure subset of $B_i$ is finite in $\pi_1(B_r(z), z)$.

The ball $B'_i$ is smooth, thus it has finite diameter $d_i$. 
The group of homeomorphisms of a ball is connected.
Every path inside $B_i$ can be free $\epsilon$-homotoped  
to a path in $B'_i$ and hence to a path whose Riemannian length is less than the diameter $d_i$. 
Thus $\gamma'(h_i,x)$ can be represented by a~path whose Riemannian length 
is less than $d_i+2(r_i+\epsilon)$, where $r_i$ is the radius of the geodesic ball $B_{r_i}(z)$ which contains $B_i$. 
By Milnor-Svarc lemma \cite{MR1744486} the word length of $\gamma'(h_i,x)$ is bounded in $\pi_1(B_r(z), z)$.
\end{proof}

We have the following commutative diagram
$$
\begin{tikzcd}
	\OP{H}^\bullet(\pi_M) \arrow[r, "\Gamma"] & \OP{H}^\bullet(\C T_{M})\\
	\OP{H}_b^\bullet(\pi_M) \arrow[r, "\Gb"] \arrow[u] & \OP{H}_b^\bullet(\C T_M) \arrow[u]
\end{tikzcd}
$$
It follows that $\Gb$ can be restricted to the exact part of the bounded cohomology.
$$
\EGb(\C T_M) \colon \OP{EH}_b^\bullet(\pi_M) \to \OP{EH}_b^\bullet(\C T_M).
$$

\begin{remark}\label{r:quasi}
We would like to point out that $\OP{EH}_b^2(G)$ is the space of non-trivial 
homogeneous quasimorphisms on $G$ \cite[Chapter 2]{Calegari}, and
$\EGb(\C T_M) \colon \OP{EH}_b^2(\pi_M) \to \OP{EH}_b^2(\C T_M)$
is the map defined by Polterovich \cite{MR2276956}.
\end{remark}

\subsection{The reduced bounded cohomology}
It is straightforward to see, that $\Gb \colon \OP{H}_b^\bullet(\pi_1(M)) \to \OP{H}_b^\bullet(\C H_M)$ is a~contraction.
Hence it defines a~map on the reduced bounded and the reduced exact bounded cohomology.
$$\rGb(\C T_M) \colon \OP{\overbar{H}}_b^\bullet(\pi_1(M)) \to \OP{\overbar{H}}_b^\bullet(\C T_M),$$
$$\rEGb(\C T_M) \colon \OP{\overbar{EH}}_b^\bullet(\pi_M) \to \OP{\overbar{EH}}_b^\bullet(\C T_M).$$
 
\section{Proofs}\label{s:proofs}

Let $M$ be a~complete Riemannian manifold of finite volume,
and $\C T_{M}$ is either $\OP{Homeo}_0(M,\mu)$ or $\OP{Diff}_0(M,\OP{vol})$.
If, in addition $M$ is a~symplectic manifold, then $\C T_{M}$ may also be $\OP{Symp}_0(M,\omega)$. 
Throughout this section we consider the map $\rEGb$ induced by $\gamma = \gamma_s$, 
where $s$ is the section defined in Subsection~\ref{s:example}.

First, let us present an outline of both proofs.
Assumptions of Theorem \ref{t:surjects.on.F} and Theorem \ref{t:hyp.embed} imply 
that there is an embedding $i \colon \OP{F}_2 \to \pi_M$ such that 
in both cases $i^* \colon \rEH_b^\bullet(\pi_M) \to \rEH_b^\bullet(\OP{F}_2)$ is surjective. 
Indeed, in Theorem \ref{t:surjects.on.F} it is straightforward, and
in Theorem \ref{t:hyp.embed} we use the result presented in \cite{fps}, 
which in particular, implies that if $\OP{F}_2 \times K$ 
is hyperbolically embedded in $\pi_M$, 
then one can extend a class in $\rEH_b^\bullet(\OP{F}_2 \times K)$ to a class in $\rEH_b^\bullet(\pi_M)$. 
That is why we require $\OP{F}_2 \times K$ to be hyperbolically embedded. 
Given an element $c\in\rEH_b^\bullet(\pi_M)$, we look at the restriction $\rEGb(c)_{|\OP{F}_2}$, 
where the group $\OP{F}_2$ is carefully embedded in $\C T_M$. 
The construction of the embedding is made in a way 
such that there is a~non-zero real number $\Lambda$ so that $\Lambda i^*(c)$ 
and $\rEGb(c)_{|\OP{F}_2}$ are close in the norm. 
It follows that the norm of $\rEGb(c)$ is positive whenever the norm of $i^*(c)$ is positive.

Let $i \colon~\OP{F}_2~\to~\pi_M$ be an embedding and let $a$ and $b$ generate $\OP{F}_2$. 
A~loop $\alpha$ in $M$ based at $z$ represents in a natural way an element of $\pi_M$
(as the $\OP{Z}(\pi_1(M))$-coset of the homotopy class of $\alpha$). 
If $\OP{dim}(M) = 2$ we assume that $i(a)$ and $i(b)$ 
are represented by simple loops based at $z$. In the next lemma we construct 
a~family of maps $\rho_\epsilon \colon \OP{F}_2 \to \C T_M$ such that the diagram

$$
\begin{tikzcd}
	\rEH^\bullet_b(\pi_M) \arrow[r, "\rEGb"] \arrow[d,"i^*"] &  
	\rEH^\bullet_{b}(\C T_{M}) \arrow[dl,"\rho_\epsilon^*"]\\
	\rEH^\bullet_b(\OP{F}_2) 
\end{tikzcd}
$$
is commutative up to scaling and a~small error controlled by $\epsilon$.

\begin{lemma}\label{l:rep}
Assume that $M$, $\C T_{M}$ and $i \colon \OP{F}_2 \to \pi_M$ are as above. 	
Then there exists a~family of homomorphisms $\rho_\epsilon \colon \OP{F}_2 \to \C T_{M}$, indexed by 
$\epsilon \in (0,1)$, satisfying the following property:
there exists a~non-zero real number $\Lambda$, such that for every class $c \in \rEH_b^\bullet(\pi_M)$ we have
$$\n\rho_\epsilon^*\rEGb (c)-\Lambda i^*(c)\n\xrightarrow{\epsilon \to 0} 0.$$
\end{lemma}

\begin{proof}
Let $\OP{dim}(M)=m$. Denote by $B^{m-1} \subset \mathbb{R}^{m-1}$ the $m-1$ dimensional closed unit ball, 
and let $S^1 = \mathbb{R}/\mathbb{Z}$. 
We fix $\epsilon \in (0,1)$ and define an isotopy $P^t_\epsilon \in \OP{Diff}(S^1 \times B^{m-1})$ by  
$$P^t_\epsilon(\psi,x)=(\psi+tf(\n x\n),x),$$ 
where $t \in [0,1]$, and $f:[0,1]\to \mathbb{R}$ is a~smooth function such that 
$f(y) = 1$ for $y \leq 1-\epsilon$ and $f(1) = 0$.
We call $P^t_\epsilon$ the finger-pushing isotopy and $P^1_\epsilon$ the finger-pushing map. 
Note that $P^0_\epsilon = Id$ and that $P^1_\epsilon$ fix point-wise the boundary of $S^1 \times B^{m-1}$
and fix all points $(\psi,x)$ for which $\n x \n \leq 1-\epsilon$. 

Let $g_0$ be the product of the standard euclidean Riemannian metrics on $B^{m-1}$ and $S^1$.
By the theorem of Fubini, the measure induced by $g_0$~is preserved 
by the map $P^t_\epsilon$ for every $t \in [0,1]$ and every $\epsilon~\in~(0,1)$.
In case $m=2k$, we similarly construct the finger-pushing 
isotopy $P^t_\epsilon \in \OP{Diff}(S^1\times B^1\times B^{2k-2})$, $t \in [0,1]$
which preserves the standard symplectic form 
$dx\wedge dy+\sum_{i=1}^{k-1}dp_i\wedge dq_i$ on $S^1\times B^1\times B^{2k-2}$. 
The precise construction is presented in \cite[proof of Theorem 1.3]{BK-frag}.  

Recall that $a,b$ are generators of $\OP{F}_2$. 
We represent $i(a)$ and $i(b)$ by embedded loops $\alpha$ 
and $\beta$ in $M$ which are based at $z$ and intersect only at $z$. 
Note that if $m=2$ this is our assumption, and
if $m>2$ then any two elements of $\pi_M$ may be represented in this way.
Let $N(\alpha)$~be a~closed tubular neighborhood of $\alpha$ and let $P^t_{\epsilon}(\alpha)$~be the isotopy
defined by pulling-back $P^t_{\epsilon}$ via $n_\alpha \colon N(\alpha) \to S^1 \times B^{m-1}$ and extending it
by the identity outside $N(\alpha)$. 
If $\C T_M = \OP{Homeo}_0(M,\mu)$, or $\OP{Diff}_0(M,\OP{vol})$, then the Moser trick allows us to choose $n_\alpha$  
such that $P^t_{\epsilon}(\alpha)$ preserves the volume form, and
if $\C T_M = \OP{Symp}_0(M,\omega)$, then the Moser trick allows us to choose $n_\alpha$  
such that $P^t_{\epsilon}(\alpha)$ preserves the symplectic form $\omega$.
Let 
$$
A_{\epsilon}(\alpha) = n_\alpha^{-1}(\{(\psi,x)~|~\n x \n\leq 1-\epsilon\}),\qquad
B_{\epsilon}(\alpha) = N(\alpha)-A_{\epsilon}(\alpha).
$$
Note that $P^1_\epsilon(\alpha)$ fixes point-wise $A_{\epsilon}(\alpha)$.
In the same way we define $P^t_{\epsilon}(\beta)$, $A_{\epsilon}(\beta)$ and $B_{\epsilon}(\beta)$.
The homomorphism $\rho_{\epsilon} \colon \OP{F}_2 \to \C T_M$  is given~by: 
$$\rho_{\epsilon}(a) = P^1_{\epsilon}(\alpha),\qquad
\rho_{\epsilon}(b) = P^1_{\epsilon}(\beta).$$
Now we show that there exists a non-zero real number $\Lambda$ such that for every $[c]\in\rEH_b^\bullet(\pi_M)$
we have
$$\n\rho_\epsilon^*\rEGb ([c])-\Lambda i^*([c])\n\xrightarrow{\epsilon\to 0} 0.$$ 
To simplify the notation, we identify $\OP{F}_2$ with its image $i(\OP{F}_2)$.
First we consider the values of $\gamma$ on elements of the form $\rho_{\epsilon}(w)$, where $w \in \OP{F}_2$.
Let $h_a \colon \OP{F}_2 \to \B \langle a \rangle$ be the 
retraction onto the subgroup generated by $a$ that sends $b$ to the trivial element.
Similarly, we define $h_b \colon \OP{F}_2 \to \B \langle b \rangle$.
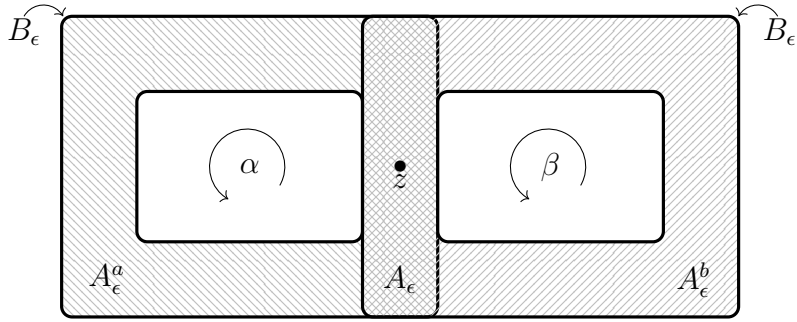
\begin{center}
\begin{tikzpicture}
	\draw[pattern=north west lines, pattern color=gray!50!white, rounded corners, very thick, ] (-4,-2) rectangle (1,2);
	\fill[white, rounded corners] (-3,-1) rectangle (0,1);
	\draw[rounded corners, very thick] (-3,-1) rectangle (0,1);
	\draw[pattern=north east lines, pattern color=gray!50!white, rounded corners, very thick] (0,-2) rectangle (5,2);
	\fill[white, rounded corners] (1,-1) rectangle (4,1);
	\draw[rounded corners, very thick] (1,-1) rectangle (4,1);

	\draw (0.5,0) node {$\bullet$};
	\draw (0.5,-0.2) node {$z$};
	\draw (0.5,-1.5) node {$A_\epsilon$};
	\draw (-3.4,-1.5) node {$A^a_\epsilon$};
	\draw (4.4,-1.5) node {$A^b_\epsilon$};
	\draw (-1.5,0) node {$\alpha$};
	\draw (2.5,0) node {$\beta$};
	\draw [->] (-1.10,-0.25) arc [radius=0.5, start angle=-30, end angle=240];
	\draw [->] (2.90,-0.25) arc [radius=0.5, start angle=-30, end angle=240];
	\draw (-4.5,1.8) node {$B_\epsilon$};
	\draw (5.55,1.8) node {$B_\epsilon$};
	\draw [->] (-4.5,2) arc [radius=0.3, start angle=150, end angle= 30];
	\draw [<-] (5,2) arc [radius=0.3, start angle=150, end angle= 30];
\end{tikzpicture}
\end{center}

From the description of $\gamma$ in Subsection \ref{s:example}, 
we see that if $x$ belongs to the set $A_\epsilon\defeq A_{\epsilon}(\alpha)\cap A_{\epsilon}(\beta)$, 
then $\gamma(\rho_{\epsilon}(w),x)$ is conjugated to $w$.
Similarly if $x\in A^a_\epsilon \defeq A_{\epsilon}(\alpha) - N(\beta)$, 
then $\gamma(\rho_{\epsilon}(w),x)$ is conjugated to $h_a(w)$ and 
if  $x\in A^b_\epsilon\defeq A_{\epsilon}(\beta) - N(\alpha)$, 
then $\gamma(\rho_{\epsilon}(w),x)$ is conjugated to $h_b(w)$. 
If $x\in B_\epsilon \defeq B_{\epsilon}(\alpha)\cup B_{\epsilon}(\beta)$, 
then we do not have any control over the loops we get,
but this case is negligible if $\epsilon$ is small enough. 
To sum up, we have:
$$
\gamma(\rho_{\epsilon}(w),x) = 
\begin{cases}
	e & x \in M-(N(\alpha) \cup N(\beta)),\\
	u_xwu_x^{-1} & x \in A_{\epsilon} = A_{\epsilon}(\alpha) \cap A_{\epsilon}(\beta),\\
	u_{a,x}h_a(w)u_{a,x}^{-1} & x \in A^a_\epsilon = A_{\epsilon}(\alpha) - N(\beta),\\
	u_{b,x}h_b(w)u_{b,x}^{-1} & x \in A^b_\epsilon = A_{\epsilon}(\beta) - N(\alpha),\\
	? & x \in B_\epsilon = B_{\epsilon}(\alpha) \cup B_{\epsilon}(\beta),
\end{cases}
$$
for some $u_x, u_{a,x}, u_{b,x} \in \pi_M$.
Let $n\in\mathbb N$ and $[c] \in \rEH_b^n(\pi_M)$.
Without loss of generality, we assume that $c(e,\ldots,e)~=~0$. 
Let $\overbar{g}~=~(g_0,g_1,\ldots)\in~\C T_{M}^n$. Denote 
$$\gamma(\overbar{g},x)=(\gamma(g_0,x),\gamma(g_1,x),\ldots).$$ 
Let $\overbar{w} \in \OP{F}_2^n$. We have: 
$$\rho_{\epsilon}^*\rEGb(c)(\overbar{w})=\rEGb(c)(\rho_\epsilon(\overbar{w}))=\int_{M} c(\gamma(\rho_{\epsilon}(\overbar{w}),x))d\mu(x).$$
Denote $u.c(\overbar{w}):=c(u\overbar{w}u^{-1})$, where $u\in\pi_M$.
Thus we obtain
\begin{align*}\label{e:split} 
	\rho_{\epsilon}^* \rEGb (c)(\overbar{w}) =& \int_{A_\epsilon}u_x.c(\overbar{w})d\mu(x) +
	\int_{A^a_\epsilon}u_{a,x}.c(h_a(\overbar{w}))d\mu(x)+\\
	&+ \int_{A^b_\epsilon}u_{b,x}.c(h_b(\overbar{w}))d\mu(x)+
	\int_{B_{\epsilon}}c(\gamma(\rho_{\epsilon}(\overbar{w}),x))d\mu(x).
\end{align*}

Recall that conjugation acts trivially on the cohomology, which gives us $[u.c]=[c]$. 
Both $\OP{Z}_b^n(G) = \OP{ker}(d_n)$ 
and $\rH_b^n(G)$ are Banach spaces and $[\cdot] \colon \OP{Z}_b^n(G) \to \rH_b^n(G)$ is a~continuous linear map.
Hence 
\begin{align*}
\left[\int_{A_\epsilon}u_x.c(\overbar{w})d\mu(x)\right]&=
\left[\sum_{u \in \pi_M}\mu(\{x \in A_\epsilon~|~u_x=u\})u.c(\overbar{w})\right]\\
&=\hspace{3px}\sum_{u \in \pi_M}\mu(\{x \in A_\epsilon~|~u_x=u\})i^*[u.c]=\mu(A_\epsilon)i^*([c]).
\end{align*}
Let $u.c_{|a}$ be the restriction of $u.c$ to the subgroup generated by the generator $a$.
The function $\overbar{w} \to c(uh_a(\overbar{w})u^{-1})$ equals to the pull-back of the cocycle $u.c_{|a}$, namely: 
$$c(uh_a(\overbar{w})u^{-1}) = h_a^*(u.c_{|a})(\overbar{w}).$$ 
Moreover, since $\rEH^n_b(\mathbb{Z})$ is trivial, the cocycle 
$\overbar{w} \to c(uh_a(\overbar{w})u^{-1})$ defines the trivial class in $\rEH^n_b(\OP{F}_2)$. 
It follows that 
$$
\left[\int_{A^a_\epsilon}u_{a,x}.c(h_a(\overbar{w}))d\mu(x)\right] = 
%\int_{A^a_\epsilon}[h_a^*(g_{a,x}.c)]d\mu(x) = 0,
\sum_{u \in \pi_M}\mu(\{x \in A_\epsilon~|~u_{a,x}=u\})h_a^*([u.c_{|a}]) = 0.
$$
The same holds for the integral over $A^b_\epsilon$. Let 
$$c^{\epsilon}_{res}(\overbar{w}) = \int_{B_{\epsilon}}c(\gamma(\rho_{\epsilon}(\overbar{w}),x))d\mu(x).$$
Note that $c^{\epsilon}_{res}$ is a~cocycle on $\OP{F}_2$. 
Now we can write:
$$
\rho^*_{\epsilon} \rEGb ([c]) = \mu(A_{\epsilon})i^*([c]) + [c^{\epsilon}_{res}],
$$
and  
$$\n [c^{\epsilon}_{res}]\n \leq \mu(B_{\epsilon})\n c\n_{sup}.$$
Moreover, $\mu(A_{\epsilon}) \xrightarrow{\epsilon \to 0} \mu(N(\alpha) \cap N(\beta))\neq 0$ and
$\mu(B_{\epsilon})~\xrightarrow{\epsilon \to 0}~0$. It follows that:
\begin{align*}
&\n \rho^*_{\epsilon} \rEGb ([c])-\mu(N(\alpha)\cap N(\beta))i^*([c])\n \leq\\
&\leq [\mu(A_{\epsilon})-\mu(N(\alpha)\cap N(\beta))]\n i^*([c])\n +\mu(B_{\epsilon})\n c\n_{sup} . 
\end{align*}
Hence $\n\rho^*_{\epsilon}\rEGb ([c])-\mu(N(\alpha)\cap N(\beta))i^*([c])\n\xrightarrow{\epsilon \to 0} 0.$ 
\end{proof}

\begin{remark}
In what follows we apply Lemma \ref{l:rep} for injective $i$. 
However, Lemma \ref{l:rep} holds for every $i$.
Injectivity was used only to simplify the notation, 
when we identified $\OP{F}_2$ with its image $i(\OP{F}_2)$.
\end{remark}

%Now we prove Theorem \ref{t:surjects.on.F} and Theorem \ref{t:hyp.embed}.

\begin{theorem}\label{t:surjects.on.F.strong}
Let $h\colon\pi_1(M)\to\OP{F}_2$ be a surjective homomorphism. 
We have 
$$\rEH_{b}^{\bullet}(\OP{F}_2) \hookrightarrow  \rEH_{b}^{\bullet}(\C T_{M}).$$
\end{theorem}

\begin{proof}
First note that since $h$ is onto, the center of $\pi_1(M)$ is 
mapped into the center of $\OP{F}_2$, hence $h$ is trivial on $\OP{Z(\pi_1(M))}$.
It means, that $h$ induces a surjective homomorphism $p \colon \pi_M \to \OP{F}_2$. 

Recall that $\OP{dim}(M)=m$.
If $m>3$, then we take $i\colon\OP{F}_2\to\pi_M$ to be any section of $p$.  
If $m=2$, then it is easy to find two embedded loops based at $z$ and intersecting only at $z$, such that
they generate $\OP{F}_2$ and there is a~retraction $\pi_1(M)\to\OP{F}_2$.
If this is the case, we substitute $p$ by this retraction (note that in this case $\pi_1(M) = \pi_M$).

Let $i$ be a~section of this new $p$ and let $p^* \colon \rEH^\bullet_b(\OP{F}_2) \to \rEH^\bullet_b(\pi_M)$. We show that $\rEGb\circ p^*$ is an embedding. 
The section $i$ satisfies the assumptions of Lemma \ref{l:rep}.
Let $\{\rho_\epsilon\}$ be the family of homomorphisms from Lemma \ref{l:rep}.
We have
$$
\begin{tikzcd}
	\rEH^\bullet_b(\pi_M) \arrow[r, "\rEGb"] \arrow[d,shift left=.75ex, "i^*"]&  
	\rEH^\bullet_{b}(\C T_{M}) \arrow[ld,"\rho_\epsilon^*"]\\
	\rEH^\bullet_b(\OP{F}_2) \arrow[u,shift left=.75ex,"p^*"] 
\end{tikzcd}
$$
Note that $i^*\circ p^*=id$.
Suppose that $d~\in~\rEH^\bullet_b(\OP{F}_2)$ is a~non-trivial class. 
In the reduced cohomology it means that $\n d\n>~0$.
Let $c=p^*(d)$. We have $\n i^*(c)\n=\n d\n>0$. 
Since $\n\rho_\epsilon^*\rEGb(c)-\Lambda i^*(c)\n\xrightarrow{\epsilon\to 0} 0$,
then for some small $\epsilon$ we have $\n\rho_\epsilon^*\rEGb(c)\n>0$.
It follows that 
$$\rEGb (c)=\rEGb(p^*(d))\neq 0.$$
Thus $\rEGb\circ p^*$ is an embedding. 
\end{proof}

\begin{remark}\label{r:int.volume}
One can define a non-trivial class in $\rEH_b^3(\OP{F}_2)$ by choosing an isometric action $\rho$ 
of $\OP{F_2}$ on the $3$-dimensional hyperbolic space $\B H^3$
and defining a cocycle $\OP{vol}_{\rho}(a_1,\ldots,a_4)$ to be the 
signed volume of the geodesic simplex $\Delta(\rho(a_1)x,\ldots,\rho(a_4)x)$, 
where $x \in \B H^3$. For some $\rho$, the class 
defined by $\OP{vol}_{\rho}$ has positive norm, see \cite{soma}.
The classes in $\rEH_b^3(\C T_M)$ which are constructed in 
Theorem \ref{t:surjects.on.F.strong} have similar geometrical interpretation. 
More precisely, the value of $\rEGb(p^*(\OP{vol}_{\rho}))(f_1,\ldots,f_4)$ is the average value of the
signed volumes of $\Delta(p\gamma(f_1,x),\ldots,p\gamma(f_4,x))$ over $M$. 
Since every $\gamma(f_i,x)$ takes essentially 
finitely many values, this average is a finite sum
of weighted signed volumes of certain simplices in $\B H^3$. 
%Recall that the signed volume equals $\OP{vol}(\Delta)$ if $\Delta$ 
%has the same orientation as $\B H^3$ and $-\OP{vol}(\Delta)$ otherwise. 
\end{remark}

\begin{theorem}
Let $K$ be a finite group and $j \colon \OP{F}_2 \times K \to \pi_M$ be a~hyperbolic embedding.
Then
$$
\OP{dim} \rEH_{b}^\bullet(\C T_{M}) \geq \OP{dim} \rEH_{b}^\bullet(\OP{F}_2).
$$
\end{theorem}

\begin{proof}
Let $\OP{dim}(M)=2$. In this case if $\OP{F}_2$ embeds in $\pi_M$, then one can find a~retraction
$\pi_M \to \OP{F}_2$. Thus if $\OP{dim}(M)=2$, the statement follows from Theorem \ref{t:surjects.on.F.strong}. 

Now suppose that $\OP{dim}(M)>2$, and let $i\colon\OP{F_2}\to\pi_M$ be the homomorphism $j$ restricted to $\OP{F}_2\times\{e\}$.
Since $\OP{dim}(M)>2$, $i(a)$ and $i(b)$ 
may be represented by based loops whose intersection is the base-point $z$. 
Let $\{\rho_\epsilon\}_\epsilon$ be the family of maps 
$\rho_\epsilon\colon\OP{F}_2\to\C T_{M}$ constructed in Lemma \ref{l:rep}.
We have
$$
\begin{tikzcd}
	\rEH^\bullet_b(\pi_M) \arrow[r, "\rEGb"] \arrow[d,"i^*"] &
	\rEH^\bullet_{b}(\C T_{M}) \arrow[ld,"\rho_\epsilon^*"]\\
	\rEH^\bullet(\OP{F}_2)
\end{tikzcd}
$$
Let us show that $\OP{ker}(\rEGb)\hspace{-2.5px}\subset\hspace{-1.5px}\OP{ker}(i^*)$. 
There is a~non-zero real number $\Lambda$ such that for every $c\in\rEH_b^\bullet(\pi_M)$
we have 
$$\n \rho_\epsilon^*\rEGb (c)-\Lambda i^*(c)\n\xrightarrow{\epsilon \to 0} 0.$$
Let $c\in\rEH_b^\bullet(\pi_M)$ be such that $\rEGb(c)=0$. Then 
$$\n\Lambda i^*(c)\n=\n \rho_\epsilon^* \rEGb (c)-\Lambda i^*(c)\n\xrightarrow{\epsilon \to 0} 0.$$ 
Hence $\n i^*(c)\n=0$ and $c\in\OP{ker}(i^*)$. 
Thus $\OP{ker}(\rEGb)\hspace{-2.5px}\subset\hspace{-1.5px}\OP{ker}(i^*)$ and
\begin{align*}
	\OP{dim} \rEH_{b}^\bullet(\C T_{M}) &\geq \OP{dim}(\rEH_{b}^\bullet(\pi_M)/\OP{ker}(\rEGb))\\
	&\geq  \OP{dim} (\rEH_{b}^\bullet(\pi_M)/\OP{ker}(i^*)).
\end{align*}
Recall that $j\colon\OP{F}_2\times K\to\pi_M$ 
is a~hyperbolic embedding. It follows from \cite{fps} that the map
$j^*\colon\rEH_b^\bullet(\pi_M)\to\rEH_b^\bullet(\OP{F}_2 \times K)$ is surjective.
Using the identification $\rEH_b^\bullet(\OP{F}_2 \times K)=\rEH_b^\bullet(\OP{F}_2)$
we can write that $i^*=j^*$. Thus $i^*$ is surjective and 
$\rEH_{b}^\bullet(\pi_M)/\OP{ker}(i^*)=\rEH_{b}^\bullet(\OP{F}_2)$.
\end{proof}

%\begin{remark}\label{r:ham}
%We have that bounded and reduced bounded cohomology of $\OP{Ham}(M)$ is the same as $\OP{Sympl}_0(M)$. 
%Indeed \textcolor{red}{finish it}.
%\end{remark}

\section {Questions and final remarks}
\begin{remark}\label{r:general} 
Versions of Theorem \ref{t:surjects.on.F} and Theorem \ref{t:hyp.embed}, where $\rEH_b^\bullet$
is substituted by $\rH_b^\bullet$ hold in a more general setting. 
Namely, they hold for a~topological manifold $M$ equipped with  
a regular finite Borel measure $\mu$ which is positive 
on open sets and zero on nowhere dense sets, 
and $\C T_M$~is the identity component of the group of  
measure-preserving homeomorphisms of $M$. 
One can also take $\C T_M$ to be the identity component of the group of volume-preserving diffeomorphisms,
or symplectomorphisms of $M$.\end{remark} 

\begin{remark}\label{r:extend}
The map $\EGb \colon \OP{EH}_b^\bullet(\pi_M)\to\OP{EH}_b^\bullet(\OP{Diff}_0(M,\mu))$ 
does not factor through $\OP{Diff}_0(M)$ (note that from this it follows that the map $\EGb \colon \OP{EH}_b^\bullet(\pi_M)\to\OP{EH}_b^\bullet(\OP{Homeo}_0(M,\mu))$ 
as well does not factor through $\OP{Homeo}_0(M)$). 
Indeed, if $\EGb$ would factor through $\OP{Diff}_0(M)$, one could construct a~non-trivial homogeneous quasimorphisms on $\OP{Diff}_0(M)$, 
which leads to a contradiction, since for many $M$ the group $\OP{Diff}_0(M)$ does not admit such quasimorphisms. 
More precisely, let $M$ be a~closed connected hyperbolic $3$-manifold.
Recall that $\OP{EH}_b^2(G)$ is the space of homogeneous quasimorphisms on $G$.
Since $\pi_M$ is non-elementary hyperbolic, we have $\OP{EH}_b^2(\pi_M)\neq 0$ and it follows from \cite[Theorem 1.11]{MR2509711} that $\OP{EH}_b^2(\OP{Diff}_0(M)) = 0$.
It is easy to see that 
$$\EGb \colon \OP{EH}_b^2(\pi_M)\to\OP{EH}_b^2(\OP{Diff}_0(M,\mu))$$
is an embedding (in the case when $M$ is a~surface, see \cite[Theorem 2.5]{eq} for the proof), 
and hence cannot factor through the trivial group.
\end{remark}

\begin{remark}
In this paper we assume that homeomorphisms are isotopic to the identity.
This assumption can be dropped 
if we substitute the group $\pi_1(M)$ by the mapping class group $\OP{MCG}(M,z)$ 
(such approach was used in \cite{eq} for surfaces). 
Indeed, let 
$$\delta^{ext} \colon \OP{Homeo}(M,z) \to \OP{MCG}(M,z)$$
be the quotient map, where 
$\OP{MCG}(M,z) = \pi_0(\OP{Homeo}(M,z))$ and let $\OP{Homeo}(M,z)$ 
be the group of homeomorphisms of $M$ fixing $z$. 
Consider the cocycle 
$$\gamma^{ext} \colon \OP{Homeo}(M,\mu) \times M \to \OP{MCG}(M,z)$$
given by $\gamma^{ext}(g,x) = \delta^{ext}(s_{g(x)}^{-1} \circ g \circ s_x)$.
This cocycle induces the map
$$\Gb^{ext} = \OP{Ind}_b(\gamma^{ext}) \colon \OP{H}_b^\bullet(\OP{MCG}(M,z)) \to \OP{H}_b^\bullet(\OP{Homeo}(M,\mu)).$$
The disadvantage of this approach is that almost nothing is known 
about $\OP{H}_b^\bullet(\OP{MCG}(M,z))$ when the dimension of $M$ is greater than $2$.  
\end{remark}

%At the end we would like to state a~question concerning certain 
%well-known higher dimensional bounded cohomology classes.

We finish this section with a question. Let $M^n$~be a~compact Riemannian manifold with negative sectional 
curvature, and let $\pi_1(M)$ act by deck-transformations on the universal cover $\wt{M}$. 
It is known that there is a~common bound for volumes of geodesic simplices in $\wt{M}$,
thus one can define a~non-trivial class $[\OP{vol}_M]\in\OP{H}_b^n(\pi_1(M))$ in a similar way as in Remark \ref{r:int.volume}. 

\begin{question}
Is the class $\Gb([\OP{vol}_M])\in \OP{H}_b^n(\OP{Homeo}_0(M,\mu))$ non-trivial?
\end{question}

%\nocite{*}
\bibliography{bibliography}
\bibliographystyle{plain}
 
\end{document}